\documentclass[12pt]{amsart}
\usepackage{amssymb} 
\usepackage{amsmath} 
\usepackage{amscd}
\usepackage{amsbsy}
\usepackage{comment}
\usepackage[matrix,arrow]{xy}
\usepackage{hyperref}
\usepackage{amsrefs}
\usepackage{tikz-cd}
\usepackage{lineno} 

\makeatletter
\@namedef{subjclassname@2010}{%
  \textup{2010} Mathematics Subject Classification}
\makeatother

\usepackage[T1]{fontenc}

\newtheorem{thm}{Theorem}[section]
\newtheorem{lem}{Lemma}[section]
\newtheorem{prop}[lem]{Proposition}

\newtheorem{cor}[lem]{Corollary}
\newtheorem{conj}{Conjecture}[section]

\newtheorem{defn}{Definition}[section]

\theoremstyle{definition}

\theoremstyle{remark}

\DeclareMathOperator{\Tr}{Trace}
\DeclareMathOperator{\Frob}{Frob}
\DeclareMathOperator{\End}{End}

\DeclareMathOperator{\Aut}{Aut}

\DeclareMathOperator{\GL}{GL}
\DeclareMathOperator{\SL}{SL}

\DeclareMathOperator{\Res}{Res}

\DeclareMathOperator{\Gal}{Gal}
\DeclareMathOperator{\norm}{Norm}

\newcommand{\Q}{{\mathbb Q}}
\newcommand{\Z}{{\mathbb Z}}
\newcommand{\C}{{\mathbb C}}
\newcommand{\R}{{\mathbb R}}
\newcommand{\F}{{\mathbb F}}

\newcommand{\cH}{\mathcal{H}}

\newcommand{\cM}{\mathcal{M}}
\newcommand{\cN}{\mathcal{N}}

\newcommand{\cO}{\mathcal{O}}

\newcommand{\cW}{\mathcal W}

\newcommand{\cm}{\mathfrak{m}}
\newcommand{\ff}{\mathfrak{f}}

\newcommand{\fq}{\mathfrak{q}}
\newcommand{\mP}{\mathfrak{P}}
\newcommand{\ml}{\mathfrak{l}}

\newcommand{\Sprime}{S^{\prime}}
\newcommand{\Lprime}{L^{\prime}}
\newcommand{\rhobar}{\overline{\rho}}

\newcommand{\hooklongrightarrow}{\lhook\joinrel\longrightarrow}

\begin{document}


\baselineskip=17pt


\title[]{Irreducible binary cubics and the generalized superelliptic equation over number fields}

\author{George C. \c Turca\c s}
\address{Mathematics Institute\\
	University of Warwick\\
	Coventry\\
	CV4 7AL \\
	United Kingdom}

\email{g.c.turcas@warwick.ac.uk}
\thanks{The author is supported by
	an EPSRC PhD studentship. }

\date{}

\begin{abstract}
	For a large class (heuristically most) of irreducible binary cubic forms $F(x,y) \in \Z[x,y]$, Bennett and Dahmen proved that the generalized superelliptic equation $F(x,y)=z^l$ has at most finitely many solutions in $x,y \in \Z$ coprime, $z \in \Z$ and exponent $l \in \Z_{\geq 4} $. Their proof uses, among other ingredients, modularity of certain mod $l$ Galois representations and Ribet's level lowering theorem. The aim of this paper is to treat the same problem for binary cubics with coefficients in $\cO_K$, the ring of integers of an arbitrary number field $K$, using by now well-documented modularity conjectures. 
\end{abstract}



\maketitle

\section{Introduction}
In their extraordinary paper Bennett and Dahmen \cite{bendah} proved that for a large class of binary forms $F \in \Z[X,Y]$ of degrees $3,4,6$ and $12$, including \lq\lq most" cubic forms (see \cite{bendah}*{Section 12}), the generalized superelliptic equation $F(x,y)=z^l$ has finitely many solutions for $x,y,z \in \Z$, $\gcd(x,y)=1$ and $l\geq\max\{2, 7-\deg F\}$ integer.
To be precise, by attaching a family of Frey-Hellegouarch curves to putative solutions of the aforementioned equation and making essential use of modularity and level-lowering theorems due to Breuil, Conrad, Diamond, Taylor and respectively Ribet, they prove that no such solutions exists for $l$ big enough. Darmon and Granville \cite{DarmonGran} gave a descent argument and made use of Falting's Theorem to conclude that for fixed values of $l$, the equation $F(x,y)=z^l$ has finitely many solutions in coprime integers $x,y$. Together these imply the result stated above. 

Modular methods are undoubtedly an extremely powerful tool for proving that certain Diophantine equations have no solutions and, in some cases, finding the set of all solutions to these equations over $\Z$ (or $\Q$). Some number theorists are therefore interested in extending these methods over more general number fields. Such attempts were successfully carried out for the Fermat equation over certain totally real number fields by Jarvis \cite{jarvis2004fermat} and by Freitas and Siksek \cite{FreitasSiksek}, \cite{AFreitasSiksek}. These rely essentially on modularity lifting theorems over totally real fields due to Barnett-Lamb, Breuil, Diamond, Gee, Geraghty, Kisin, Skinner, Taylor, Wiles and others.

On the other hand, modularity of elliptic curves over number fields with complex embeddings is highly conjectural. Nevertheless, assuming by now well-documented conjectures in the Langlands programme, \c{S}eng\"{u}n and Siksek \cite{ASiksek} proved an asymptotic version of Fermat's Last Theorem over infinitely many general number fields.

In the spirit of \cite{ASiksek}, the purpose of this work is extend some of the results of Bennett and Dahmen \cite{bendah} to the general number field setting and to highlight the additional challenges that arise in this context.

Fix once and for all an algebraic closure $\overline{\Q}$ of $\Q$. Throughout, $l$ denotes a rational prime. Given a number field $K \subset \overline{\Q}$, we denote by $\cO_K$ its ring of integers and by $G_K = \Gal(\overline \Q/K)$ its absolute Galois group. To keep this introduction self-contained we postpone for later sections the precise statements of the two conjectures we assume. Instead, we only indicate briefly what they are.

\begin{itemize}
	\item Conjecture \ref{conj:Serre} is a version of Serre's modularity conjecture for odd, irreducible, continuous $2$-dimensional mod $l$ representations of $G_K$ that are finite flat at every prime over $l$.	
	\item Conjecture \ref{conj:ES}, sometimes referred to as \textit{Eichler-Shimura}, is part of the Langlands Programme (see \cite{taylor}) and relates weight $2$ newforms (for $\GL_2$) over $K$ that have integer Hecke eigenvalues to elliptic or fake elliptic curves over $K$.
\end{itemize} 

Before presenting our main results, we have to set up some notation. Given a number field $K$, it is known that every class in its ideal class group contains infinitely many prime ideals. If $c_1, \dots, c_h$ are the ideal classes of $K$, for every $i \in \{1, \dots, h \}$ we choose a prime ideal $\cm_i \subset \cO_K$ of smallest possible norm, such that $\cm_i \nmid 2$ and $\cm_i$ belongs to the class $c_i$. We fix the set
\begin{equation} \label{cH}
\cH_K := \left\{ \begin{array}{l} \emptyset, \text{ if } h=1 \\
\{ \cm_1, \dots, \cm_h \}, \text{ if } h \geq{2} \end{array}. \right.
\end{equation} 
Given an irreducible binary cubic $F \in \cO_K[X,Y]$ of discriminant $\Delta_F$ (one could work in greater generality and choose $F$ to be a Klein form, see \cite{bendah}), we denote by
\begin{equation}
S_{F} := \cH_K \cup \{ \text{prime ideals dividing } 2\Delta_F\} \cup \{ \text{real infinite places of } K\}.
\end{equation}
This set depends on the form $F$ (and of course, on the number field $K$).

A large part of the present paper is dedicated to proving the following result.

\begin{thm} \label{GenNf} Let $K$ be a number field for which Conjecture \ref{conj:Serre} and Conjecture \ref{conj:ES} hold. Consider $F(x,y)= \alpha_0 x^3+ \alpha_1 x^2y + \alpha_2xy^2+ \alpha_3 y^3  \in \cO_K[x,y]$ an irreducible binary cubic form such that there exists a prime ideal $\fq \mid \mid \Delta_F$ and $\fq \nmid (2 \alpha_0)$. If the Thue-Mahler equation
	\begin{equation} \label{unitseq}
	F(x,y) \in \mathcal O_{K,S_F}^{*}
	\end{equation}
	has no solutions in $x,y \in \cO_K$, then there exists a constant $A_{F} >0$ such that for all
	rational primes $l > A_{F}$ the superelliptic equation
	\begin{equation} \label{superelliptic}
	F(x,y) = z^l
	\end{equation} 
	does not have solutions in $x,y,z \in \cO_K$ such that $\gcd(x,y,z)$ is supported on the primes in $S_F$ and $\fq \nmid z$. 
\end{thm}

Proposition 2.1 of Darmon-Granville \cite{DarmonGran} implies that for any fixed value of $l \geq 4$, equation \eqref{superelliptic} has finitely many \textit{proper solutions} $x,y,z \in \cO_K$. The authors of \textit{loc.\ cit.} introduce the notion of \textit{proper solutions} to exclude the possibility of generating an infinite number in the following way. Suppose $x,y,z \in \cO_K$ are a solution to \eqref{superelliptic} and $\xi \in \cO_K^{\times}$ be a generator of the unit group. Then $\xi^{n \cdot l} x, \xi^{n \cdot l} y, \xi^{3 \cdot n} z$ for all $n \in \mathbb N$ will be an infinite family of integral solutions to our generalized superelliptic equation. A proper solution is, in fact, an equivalence class of solutions to \eqref{superelliptic} such that $\gcd(x,y)$ divides some a priori fixed ideal. Two such solutions are equivalent if we can obtain one from the other via a trivial action of the unit group $\cO_K^{\times}$.

\begin{cor}
	Let $K$ and $F$ satisfy all the hypothesis of Theorem \ref{GenNf}. The superelliptic equation
	$F(x,y)=z^l$ has finitely many \textit{proper solutions} in integers $l \geq 4$ and $x,y,z \in \cO_K$ such that $\fq \nmid z$ and the ideal $\gcd(x,y,z)$ is supported on the primes in $S_F$.	
\end{cor}

We remark that specializing to number fields of small degree and trivial class group, one could carry the proof of Theorem~\ref{GenNf} and effectively compute the constant $A_F$. In particularly fortuitous situations, one could even find oneself in positions where Conjecture \ref{conj:ES} is known to hold and therefore producing special cases of Theorem~\ref{GenNf}  that only depend on Conjecture \ref{conj:Serre}. This is emphasized in \cite{turcas2018}, where the author worked on Fermat's equation over quadratic imaginary number fields. On the other hand, the finiteness result of Darmon and Granville is obtained by appealing to Falting's theorem, hence not giving any information about the number of \textit{proper solutions} needed for making the above corollary effective.  

Over totally real fields, instead of using Serre's conjecture (see Conjecture \ref{conj:Serre}) we can take advantage of modularity theorems and prove the following more general result.

\begin{thm} \label{treal} Let $K$ be a totally real Galois number field for which Conjecture \ref{conj:ES} holds and $F \in \cO_K[x,y]$ an irreducible binary cubic. If the Thue-Mahler equation \eqref{unitseq} does not have solutions in $x,y \in \cO_K$, then there exists a constant $A_F >0$ such that for all rational primes $l>A_F$, the superelliptic equation \eqref{superelliptic} does not have solutions in $x,y,z \in \cO_K$ such that the $\gcd(x,y,z)$ is supported only on primes in $S_F$.	
\end{thm}

\noindent \textbf{Remark.} The assumption that $K$ is Galois is needed in order to prove that, for large $l$, a certain mod $l$ Galois representation is irreducible. If the number field is totally real but not Galois, it will become clear from our proof that an analogous statement to Theorem \ref{GenNf} holds independently of Conjecture \ref{conj:Serre} and assuming only Conjecture \ref{conj:ES}. In general, it is not possible to compute the constant $A_F$ introduced in the theorem above and the reason will be explained in Section \ref{secreal}.

\medskip

The insolubility of \eqref{unitseq} seems at first look very restrictive. As pointed out in \cite{bendah}, even when $K = \Q$ one has to go up to discriminant $|\Delta_F|=2063$ to find the first example of a binary cubic where the $S_F$-units equation is insoluble. We refer to Section 9 of the respective paper for an example of an infinite family of rational binary cubics satisfying the hypothesis of this theorem. In the same paper, the authors give a heuristic argument for the fact that Theorem \ref{treal} is applicable to a density one subset of the set of all rational cubic forms.

An analogous corollary to the one above follows from the last theorem when combined with the aforementioned results of \cite{DarmonGran}.

\subsection{Differences between general and totally real number fields}
Although sharing similar hypothesis and conclusions, the proofs of Theorems \ref{GenNf} and \ref{treal} are fundamentally different. We highlight here some of the most important differences.
\begin{enumerate}
	\item For a general number field $K$, Serre's modularity conjecture relates a representation $G_K \to \GL_2(\F_l)$, satisfying certain conditions, to a mod $l$ eigenform of weight $2$ over $K$. If $K$ is totally real such a mod $l$ eigenform lifts to a complex eigenform over $K$, but this is not generally the case for a number field with complex embeddings. We proceed as in \cite{ASiksek}, showing that if $l$ is sufficiently large then all mod $l$ eigenforms lift. This step makes the computation of the constant $A_F$ in Theorem \ref{GenNf} not feasible in general. To write down a formula for this constant, we would need bounds for the size of torsion subgroups of integral cohomology groups associated to locally symmetric spaces (see Section \ref{eigenfgl2}). It is maybe just worth remarking that if one chooses $K$ totally real in Theorem \ref{GenNf}, then one could compute the constant $A_F$ explicitly.
	\item In order to make the required hypothesis of Theorem \ref{treal} more general, we do not work with Serre's modularity conjecture but instead we use the known fact that for a fixed totally real field $K$, all but finitely many $\overline{K}$-isomorphism classes of elliptic curves defined over $K$ are modular (see \cite{Freitas2015}*{Theorem 5}). By increasing the value of $l$, we can make sure that the $j$-invariants of our family of Frey-Hellegouarch curves are not among the $j$-invariants of the non-modular curves. Unfortunately, this step makes the constant $A_F$ in Theorem \ref{treal} ineffective, the reason being that Theorem 5 in \text{loc.\ cit.} matches certain non-modular curves with rational points on a finite set of curves of genus $>1$ and then appeals to Faltings' theorem to deduce that there are only finitely many of them.      
	\item If $K$ has a real embedding, then a weight $2$ complex eigenform over $K$ with rational eigenvalues conjecturally (see Conjecture \ref{conj:ES}) corresponds to an elliptic curve over $K$. However, for a general number field $K$, the same conjecture predicts that such an eigenform corresponds to either an elliptic curves or a \textit{fake elliptic curve}. Following the recipe of \cite{ASiksek}, we show that the images of inertia at some fixed prime dividing $\Delta_F$ of the mod $l$ representation of our Frey-Hellegouarch curves are incompatible with images of inertia for fake elliptic curves, thereby eliminating the second possibility in our setting.      
\end{enumerate}

\section{Eigenforms for $\GL_2$ over number fields and Serre's modularity conjecture}
\label{eigenfgl2}
The exposition in this section follows closely the lines of \cite{jonessen,ASiksek}. A celebrated theorem of Khare and Wintenberger connects certain $2$-dimensional continuous representations of $G_{\Q}$ into $\GL_2(\F_{l})$ with classical modular forms of the hyperbolic plane $\cH_2$. In this section, we are about to discuss a conjecture that aims to generalise the theorem of Khare and Wintenberger in a way that is going to be soon clarified. 

Let $K$ be a number field with signature $(r,s)$ and let $G:= \GL_2(K \otimes \mathbb R)$, a real Lie group. If we denote by $A$ the diagonal embedding of $\mathbb R_{>0}$ into $G$ and by $M$ the maximal compact subgroup of $G$, the associated symmetric space (see \cite{gunnells2014}) is given by
$$D := G/AK \cong \cH_2^{r} \times \cH_3^{s} \times \mathbb R_{>0}^{r+s-1},$$
where $\cH_2,\cH_3$ are the hyperbolic plane and space respectively.   

We are going to denote by $\widehat{\mathcal O}_K, \mathbb A_K^f$ the rings of finite ad\`{e}les of $\cO_{K}$ and $K$ respectively. Fix an ideal $\cN \subset \cO_K$ and define the compact open subgroup
$$U_0(\mathcal N) := \left\{ \gamma \in \GL_2(\widehat{\cO}_K) \Big| \gamma \equiv \left( \begin{array}{cc} * & * \\ 0 & * \end{array} \right) \text{ mod } \mathcal N  \right\}.$$

Consider the adelic locally symmetric space
$$Y_0(\mathcal N) := \GL_2(K) \backslash \left( \left( \GL_2(\mathbb A^{f}_{K}) / U_0(\mathcal N) \right) \times D \right).$$ 
This space is a disjoint union $$Y_0(\mathcal N) = \bigcup_{j=1}^{h_K} \Gamma_j \backslash D,$$
where $\Gamma_j$ are arithmetic subgroups of $\GL_2(K)$ with $\Gamma_1$ being the usual congruence subgroup $\Gamma_0(\mathcal N)$ of $\GL_2(\cO_K)$ and $h_K$ the class number of $K$.

The cohomology groups $H^{i}(Y_0(\cN), \overline{\F}_l)$ come equipped with commutative Hecke algebras $\mathbb T^i_{\F_{l}}$. The latter are generated by Hecke operators $T_{\mathfrak q}$ associated to prime ideals $\fq$ of $\cO_{K}$ coprime to $l \mathcal N$.

\begin{defn} A mod $l$ eigenform $\Psi$ of level $\cN$ and degree $i$ is a ring homomorphism
	$\Psi: \mathbb T^i_{\F_{l}} \to \overline{\mathbb F}_l$.
\end{defn}

It is known that the values of a mod $l$ eigenform $\Psi$ generate a finite field extension of $\F_{l}$. We will call two mod $l$ eigenforms of levels $\cN, \cM$ equivalent if their values agree on Hecke operators associated to prime ideals away from $l \cN \cM$. It was conjectured by Calegari and Emerton in \cite{calegari2011completed} that every mod $l$ eigenform should be equivalent to one with the same level and degree $r+s$. This conjecture is known to hold for $K$ imaginary quadratic due to low virtual cohomological dimension.

Similarly, the complex cohomology groups $H^{i}(Y_0(\mathcal N), \mathbb C)$ come equipped with Hecke operators $T_{\fq}$ for all prime ideals $\fq \subset \cO_K$ not dividing $\mathcal N$. These operators generate a commutative Hecke algebra $\mathbb T^i_{\C}(\cN)$.

\begin{defn}
	A complex eigenform $f$ of level $\mathcal N$ and degree $i$ is a complex valued character of $\mathbb T^i_{\C}(\cN)$, i.e. a ring homomorphism $\mathbb T^i_{\C}(\cN) \to \C$. 
\end{defn}

Given such complex eigenform $f$, it is known that its values generate a finite extension $\Q_f$ of $\Q$. Therefore, one can fix an ideal $\ml$ of $\Q_f$ above $l$ and obtain a mod $l$ eigenform, of the same degree and level, by just setting $\Psi_f(T_{\fq}) = f(T_{\fq}) \text{ mod } \ml$, for all primes $\fq$ coprime to $l \mathcal N$. It is said that a mod $l$ eigenform $\Psi$ lifts to a complex one if there is a complex eigenform $f$ with the same level and degree such that we can obtain $\Psi$ reducing $f$ as above, i.e. $\Psi=\Psi_f$.

A complex eigenform $f$ is called trivial if $f(T_{\fq})=\norm_{\Q_f/ \Q}(\fq)+1$ for all prime ideals $\fq$ of $\cO_K$ coprime to the level. Similarly, a mod $l$ eigenform $\Psi$ is called trivial if $\Psi(T_{\fq}) \equiv \norm_{\Q_f/ \Q}(\fq)+1 \mod{l}$ for all prime ideals $\fq$ away from $l$ and the level. Every trivial mod $l$ eigenform can be obtained by reducing an Eisenstein series associated to some cusp of $Y_0(\mathcal N)$, hence they lift to complex ones. 
%
The restriction to eigenforms for $\GL_2$ made in this paper allows us to avoid giving the definion of a ``cuspidal" eigenform. In the setting of $\GL_2$, non-triviality amounts to cuspidality.  

The existence of an eigenform (complex or mod $l$) is equivalent to the existence of a class in the corresponding cohomology group that is a simultaneous eigenvector for the Hecke operators such that its eigenvalues match the values of the eigenform. We  fix an embedding from $\overline{\Q} \hooklongrightarrow \mathbb C$.
Unlike the classical situation in which $K= \Q$, when $K$ is a general number field not all mod $l$ eigenforms lift to complex ones. To explain this, let us denote by $\Z_{(l)}$ the ring of rational numbers with denominators prime to $l$.  Consider the following short exact sequence given by multiplication-by-$l$

\begin{center}
	\begin{tikzcd}	
	0 \arrow[r] & \mathbb Z_{(l)} \arrow[r, "\times l"]  & \mathbb Z_{(l)} \arrow[r]  & \mathbb F_l \arrow[r] & 0
	\end{tikzcd}.
\end{center}
This gives rise to a long exact sequence on cohomology
\begin{center}
	\begin{tikzcd}
	\dots H^i(Y_0(\mathfrak N), \mathbb Z_{(l)}) \arrow[r,"\times l"] & H^i(Y_0(\mathfrak N), \mathbb Z_{(l)}) \arrow[r] \arrow[d, phantom, ""{coordinate, name=Z}] & H^i(Y_0(\mathfrak N), \mathbb F_l) \arrow[dll,
	"\delta",
	rounded corners,
	to path={ -- ([xshift=2ex]\tikztostart.east)
		|- (Z) [near end]\tikztonodes
		-| ([xshift=-2ex]\tikztotarget.west)
		-- (\tikztotarget)}] \\ H^{i+1}(Y_0(\mathfrak N), \mathbb Z_{(l)}) \arrow[r] & \dots
	\end{tikzcd}
\end{center}
from which we can extract the short exact sequence
\begin{center}
	\begin{tikzcd}[column sep = small]
	0 \arrow[r] & H^i(Y_0(\mathfrak N), \mathbb Z_{(l)}) \otimes \mathbb F_l \arrow[r] & H^i(Y_0(\mathfrak N), \mathbb F_l) \arrow[r, "\delta"] & H^{i+1}(Y_0(\mathfrak N), \mathbb Z_{(l)})[l] \arrow[r] & 0
	\end{tikzcd}.	
	
\end{center}

In the above, the presence of $l$-torsion in $H^{i+1}(Y_0(\mathfrak N), \mathbb Z_{(l)})$ is the obstruction to surjectivity for the map $H^i(Y_0(\mathfrak N), \mathbb Z_{(l)}) \otimes \F_l \to H^i(Y_0(\mathfrak N), \mathbb F_l)$. If there is only trivial such torsion, then any Hecke eigenvector $\overline c$ in $H^{i}(Y_0(\mathfrak N), \mathbb F_l)$ comes from such an eigenvector in $H^{i}(Y_0(\mathfrak N), \mathbb Z_{(l)}) \otimes \F_l$. Using a lifting lemma of Ash and Stevens \cite{ash1986}*{Proposition 1.2.2}, we deduce that there are
\begin{enumerate}
	\item a finite integral extension $R$ of $\Z_{(l)}$
	\item  a prime $\mathfrak l$ of $R$ above $l$ and
	\item a Hecke eigenvector $c$ in $H^{i}(Y_0(\mathfrak N), R)$
\end{enumerate}
such that the Hecke eigenvalues of $c$ reduced modulo $\mathfrak l$ are equal to the ones of $\overline{c}$. Using our fixed embedding $\overline{\Q} \hookrightarrow \C$ we can regard $c$ as a class in $H^{i}(Y_0(\cN), \C)$, which implies the existence of our sought after complex eigenform. 

\medskip

\noindent \textbf{Remark.} We observe that in the paragraph above, $\overline{c}$ is not necessarily the reduction of $c$. The result that we cite only states that \textit{a system of eigenvalues} occurring in $\F_l$ may, after finite base extension, be lifted to a system occurring in $\Z_{(l)}$. The interested reader should consult \cite{ash1986}*{Section 1.2} for a more illuminating discussion.  

\medskip

In the proof of our first theorem, we will be using a special case of Serre's modularity conjecture over number fields. Serre conjectured (see \cite{Serre}) that all absolutely irreducible, odd mod $l$ Galois representations of $G_{\Q}$ arise from a classical cuspidal eigenform $f$. In the same article, Serre gave a recipe for the level $N$ and the weight $k$ of the sought after eigenform. As previously mentioned, this conjecture was proved by Khare and Wintenberger \cite{Khare2009}. We now state a conjecture concerning mod $l$ representations of $G_K$, where $K$ is an arbitrary number field.

\begin{conj}[see \cite{ASiksek}*{Conjecture 3.1}] \label{conj:Serre}
	Let $\overline{\rho} : G_K \rightarrow \GL_2(\overline{\F}_l)$ be an odd, 
	irreducible, continuous representation with Serre conductor $\cN$ (prime-to-$l$
	part of its Artin conductor) such that
	${\det}(\overline{\rho})= \chi_l$, the mod $l$ cyclotomic character. 
	Assume that $l$ is unramified in $K$ and that $\overline{\rho}\vert_{G_{K_{\mathfrak l}}}$ arises from a finite-flat group scheme over
	$\cO_{K_{\mathfrak l}}$ for every prime $\mathfrak l \mid l $. Then there is a (weight $2$) mod $l$ eigenform $\theta$ over $K$ of level
	$\cN$ such that for all primes $\mP$ coprime to $l\cN$, we have 
	\[
	\Tr(\overline{\rho}
	({\Frob}_{\mP})) = \theta(T_{\mP}). 
	\]
\end{conj}

\medskip

\noindent\textbf{Remark.} One of the hypothesis required in \cite{ASiksek}*{Conjecture 3.1} is that the prime to $l$ part of $\det(\rhobar)$ is trivial. This means that $\det(\rhobar)$ could a priori be a power of $\chi_l$, the mod $l$ cyclotomic character. The authors of \cite{ASiksek} informed us that after a discussion with Alain Krauss they have strong reasons to believe that the aforementioned conjecture should be more restrictive. This is the reason why the present Conjecture \ref{conj:Serre} is an augmented version of  \cite{ASiksek}*{Conjecture 3.1}, which requires that $\det(\rhobar)$ should be equal to $ \chi_{l}$ and not just a power of it.

\medskip

\noindent \textbf{Remark.} For every real embedding $\sigma:K \hookrightarrow \R$ and every extension $\tau : \overline{K} \to \mathbb C$ of $\sigma$, we obtain a complex conjugation $\tau^{-1} \circ c \circ \tau \in G_K$, where $c$ is the non-trivial element of $\Gal(\C / \R)$. A representation $\overline\rho:G_K \to \GL_2(\overline{\F}_l)$ is called \textit{odd} if the determinant of every complex conjugation is $-1$. In the absence of such complex conjugations (i.e. when the field $K$ is totally complex) we regard every representation as odd.	

\medskip

Although it is conjecturally easy to predict the level  $\mathcal N$  of such an eigenform, doing the same thing for the weight can be very difficult. A quite involved general weight recipe for $\GL_2$  over number fields was given by Gee et al. \cite{GeeHerzSav}. We will just mention that this recipe depends on the restriction  $\overline{\rho}|_{I_\ml}$  to the inertia subgroups for the primes   $\ml \subset \mathcal O_K$  above $l$. We only considered very special representations   $\overline{\rho}$  (that are finite flat at $\ml \mid l$ ),  for which Serre’s original weight recipe applies and predicts the trivial weight \cite{Serre}. This is why we end up with classes in   $H^{i}(Y_0(\mathcal N), \overline{\F}_l)$,  the trivial weight meaning that we get  $\overline{\F}_l$  as coefficient module.

For every complex cuspidal newform $\ff$ with rational integer Hecke eigenvalues, Langland's philosophy predicts there should be a motif attached to $\ff$. This motif is not always an elliptic curve. Some such newforms correspond to a special type of abelian surface as one can see in \cite{cremonatwist}*{Theorem 5} or \cite{AFreitasSiksek}*{Section 4}.

A simple abelian surface $A$ over $K$ whose algebra $D:=\End_K(A) \otimes_{\Z} \Q$ of $K$-endomorphisms is an indefinite division quaternion algebra over $\Q$ is called a \textit{fake elliptic curve}. It is known that if $A/K$ is a fake elliptic curve, then $K$ must be totally complex.

Let $A/K$ be a fake elliptic curve and let $l$ be a prime of good reduction for $A$. From the $l$-adic Tate module $T_l(A)$ we get a representation $\sigma_{A,l}: G_K \to \GL_4(\Z_l)$. Let $\mathcal O$ be $\End_K(A)$, viewed as an order in $D$. It is a standard fact in the theory of quaternion algebras over number fields $K$, as discussed for example in \cite{sengun3manifolds}, that all but finitely many primes $l$ split the quaternion algebra $D$. It is described in Ohta's paper \cite{ohta} that if one denotes $\cO_l := \cO \otimes_{\Z} \Z_l$, then $T_l(A)$ is isomorphic to $\cO_l$ as left $\cO$ module. The action of $\cO$ on $T_l(A)$ is via endomorphisms. This is the source of a two dimensional $l$-adic Galois representation
$$\rho_{A,l}: G_K \to \Aut_{\cO}(T_l(A)) = \cO_l^{\times} \cong \GL_2(\Z_l).$$
Moreover, the author of \textit{loc.\ cit.} proves that $\sigma_{A,l}  = \rho_{A,l} \oplus \rho_{A,l}$. 

We are now ready to state the second conjecture used in this paper.

\begin{conj}[\cite{ASiksek}*{Conjecture 4.1}]\label{conj:ES}
	Let $\ff$ be a (weight $2$) complex eigenform over $K$ of level $\mathcal{N}$ that is non-trivial, new and has rational integer Hecke eigenvalues. If $K$ has some real place, then there exists an elliptic curve $E_\ff/K$, of conductor 
	$\mathfrak{N}$, such that
	\begin{align}\label{eq: property defining E_f}
	\# E_\ff(\cO_K/\fq) = 1 + {\bf N}\fq - \ff(T_\fq) 
	\quad \text{for all $\fq \nmid \cN$}.
	\end{align}
	If $K$ is totally complex, then there exists either an elliptic curve $E_\ff$ of conductor $\mathcal N$ satisfying \eqref{eq: property defining E_f} or a fake elliptic curve $A_\ff/K$, of conductor 
	$\cN^2$, such that
	\begin{align}\label{eq: property defining A_f}
	\# A_\ff(\cO_K/\fq) = \left (1 + {\bf N}\fq - \ff(T_\fq) \right )^2 
	\quad \text{for all $\fq \nmid \cN$}.
	\end{align}
\end{conj}

For a partial result towards Conjecture \ref{conj:ES} we refer to \cite{AFreitasSiksek}*{Theorem 8}, which was derived by Blasius from the work of Hida \cite{hida81}. In particular, the aforementioned conjecture holds when $K$ is totally real such that $[K:\Q]$ is odd.

A standard fact about fake elliptic curves is the following.

\begin{thm}[\cite{Jordan1986}*{Section 3}] Let $A/K$ be a fake elliptic curve. Then $A$ has potentially good reduction everywhere. More precisely, let $\fq$ be a prime of $K$ and consider $A/K_{\fq}$. There is a totally ramified extension $K'/K_{\fq}$ of degree dividing $24$ such that $A/K'$ has good reduction.
\end{thm}

The above theorem trivially implies the following proposition, which we just record here for further reference.
\begin{prop} \label{fakeinertia} If $\overline{\rho}_{A,l}: G_K \to \mathbb \GL_2(\F_l)$ is the mod $l$ reduction of the $l$-adic representation defined above, then for every prime ideal $\fq \subseteq \cO_K$ we have $$\# \overline{\rho}_{A,l}(I_{\fq}) \leq 24,$$
	where $I_{\fq} \unlhd G_K$ is the inertia subgroup at $\fq$. 
\end{prop}
This is going to be crucial in Section \ref{seccomplex} when we are showing that mod $l$ Galois representations of elliptic curves attached to a solution of \eqref{superelliptic} are not \textit{compatible} with representations $\overline{\rho}_{A,l}$ for any fake elliptic curve $A$.

\section{Properties of the Frey curve}
The proofs of Theorem \ref{GenNf} and \ref{treal} use a construction of Bennett and Dahmen \cite{bendah}. For every Klein form (see \cite{bendah}*{Section 2} for the precise definition of these special binary forms) $F(x,y)\in \cO_K[X,Y]$, the authors of \textit{loc.\ cit} constructed a family of  Frey-Hellegouarch curves $E_{x,y}$. Important properties of this family of curves are nicely controlled by the Klein form $F(x,y)$. It turns out that all non-singular binary cubics are Klein forms (of index $2$) and from now on we restrict ourselves to these particular forms. One could write similar, more general, statements for our main two theorems when $F$ is allowed to be a general Klein form. Although these could be proved following the strategy presented here, we only treat binary cubics. In this case, the computations are shorter and easier to follow, therefore facilitating a better exposition of the main ideas in our proofs. The author is considering dedicating a chapter from his PhD thesis to analogous work on all the types of Klein forms.

Let $F(x,y) \in K[X,Y]$ be an irreducible binary cubic of discriminant $\Delta_F$. Its Hessian is the quadratic form
$$H(x,y)= \frac{1}{4} \left| \begin{array}{cc} F_{xx} & F_{xy} \\ F_{xy} & F_{yy}  \end{array} \right| $$
and the \textit{Jacobian determinant} of $F$ and $H$ is the cubic form
$$G(x,y)= \left| \begin{array}{cc} F_{x} & F_y \\ H_x & H_y \end{array} \right|. $$
Connecting them there is the following identity, vital for the construction of the family of Frey-Hellegouarch curves in \cite{bendah},
\begin{equation} \label{syzygy} 4H(x,y)^3 + G(x,y)^2 = -27 \cdot \Delta_F \cdot F(x,y)^2. \end{equation}
The identity above holds for every binary cubic $F\in K[X,Y]$ whose discriminant $\Delta_F$ is non-zero, not just the irreducible ones. There are similar identities, called \textit{syzygies}, for all types of Klein forms. Non-existence of such syzygies is the obstruction in extending this construction of such families of elliptic curves associated to all irreducible binary forms.

For future use, we record the following proposition.

\begin{prop}[\cite{bendah}*{Prop. 2.1}] \label{resultant} The resultant of a binary form $F$ of degree $k$ with its Hessian $H$ satisfies
	$$\Res(H(x,y),F(x,y))= (-1)^k \Delta_F^2.$$
\end{prop}

Suppose now that $F$ has integral coefficients, more precisely $F(x,y)=\alpha_0x^3+\alpha_1x^2y + \alpha_2 xy^2 + \alpha_3 y^3$ with $\alpha_i \in \cO_K$, for all $i \in \{0,1,2,3,4 \}$.
Set $T(x,y) = \alpha_1 x - \alpha_2 y $.
The authors of \cite{bendah} constructed the following Frey-Hellegouarch curve

\begin{equation} \label{freyxy} E_{x,y}: Y^2 = X^3+TX^2+ \frac{T^2+H}{3}X + \frac{T^3+3TH+G}{27}. \end{equation}
The dependence of $x,y$ is implicit, as $T= T(x,y)$ and $H = H(x,y)$.

\medskip

It is easy to show that $(T^2+H)/3$ and $(T^3+3TH+G)/27 \in \cO_K[X,Y]$. Making use of the formula for the discriminant of an elliptic curve and of the syzygy \eqref{syzygy},  the fundamental quantities associated to \eqref{freyxy} can be computed as
\begin{equation} \label{discrxy}
\Delta(x,y)= 2^4 \Delta_F F(x,y)^2, \hspace{0.5cm}
c_4(x,y) = - 2^4 H(x,y), \hspace{0.5cm} c_6(x,y)=-2^5 G(x,y)
\end{equation}
and
\begin{equation} \label{jxy}
j(x,y) = \frac{-2^8 H(x,y)^3}{\Delta_F F(x,y)^2}
\end{equation}

\begin{prop} \label{semaway} Let $E_{x,y}$ be a family of Frey curves associated to a binary cubic form $F \in \mathcal O_K[x,y]$, as in \eqref{freyxy}. Let $\mP \not \in S_F$ be a prime ideal in $\cO_K$ and $x_1,y_1 \in \mathcal O_K$ such that $\mP \nmid \gcd(x_1,y_1)$. Then $E_{x_1,y_1}$ is semistable at $\mP$.
\end{prop}

\begin{proof} It is known that if a Weierstrass model of $E_{x_1,y_1}$ over $\cO_K$ has $\mP \nmid c_4(x,y)$ or $\mP \nmid \Delta(x_1,y_1)$ then $E_{x_1,y_1}$ is semistable at $\mP$. Recall that the set $S_F$ contains the prime ideals dividing $2 \Delta_F$.  The proposition follows from the formulas for $\Delta(x_1,y_1)$ and $c_4(x_1,y_1)$ and the resultant identity in Proposition \ref{resultant}.
\end{proof}

\begin{prop} \label{notchangeN}
	Let $E_{x,y}$ be a family of elliptic curves defined as above. Suppose that for some $x_1,y_1 \in K$ the conductor of $E_{x_1,y_1}$ is supported only on $S_F$. If the class number of $K$ is greater than one, there exists $\xi \in K^{\times}$ such that $\xi \cdot x_1, \xi \cdot y_1$ are integral, $\gcd(\xi \cdot x_1, \xi \cdot y_1) \in \cH_K$ and the conductor of $E_{ \xi \cdot x_1, \xi \cdot y_1}$ is supported only on $S_F$.	If the class number of $K$ is one,  there exists $\xi \in K^{\times}$ such that $\xi \cdot x_1, \xi \cdot y_1$ are integral, $\gcd(\xi \cdot x_1, \xi \cdot y_1)=1$ and the same conclusion about the conductor holds. 
\end{prop}

\begin{proof} We only threat the case in which $K$ has non-trivial class group, as the proof for the latter case follows obviously from the former. As a consequence of cancelling denominators, it is obvious that we can scale the pair $(x_1,y_1)$ by some non-zero $\xi_1 \in \cO_K$ such that $\xi_1 x_1, \xi_1 y_1$ are integral. Recall that in \eqref{cH}, we defined $\cH_K = \{ \cm_1, \dots, \cm_h \} $ and therefore $[\gcd(\xi_1 x_1, \xi_1 y_1)]$ must be equal to $[\mathfrak m_i]$, for some $i \in \{ 1, \dots, h\}$.  Hence, there exists $\xi_2 \in K^{\times}$ such that $\xi_2 \cdot \gcd(\xi_1 x_1, \xi_1 y_1) = \mathfrak m_i$.
	
	We set $\xi := \xi_2 \cdot \xi_1$ and let $(x_2,y_2)= \xi \cdot (x_1, y_1) \in \mathcal O_K^2$.
	Suppose that $\mP \not\in S_F$ is a prime ideal. By the previous proposition, we know that $E_{x_2,y_2}$ is semistable at $\mathfrak P$. If $\mP$ divides the conductor of $E_{x_2,y_2}$ then it must be a prime of multiplicative reduction. This implies the fact that $v_{\mP}(j(x_2,y_2)) <0$. But since
	$$ j(x_2,y_2)= \frac{-2^8 \cdot H(x_2,y_2)^3}{\Delta_F \cdot F(x_2,y_2)^2} = \frac{-2^8 \cdot \xi^6 \cdot  H(x_1,y_1)^3 }{\Delta_F \cdot \xi^6 \cdot F(x_1,y_1)^2}= j(x_1,y_1),$$
	$\mP$ must be a prime of multiplicative reduction for $E_{x_1,y_1}$, which is a contradiction since the conductor of this curve is supported only on $S_F$.   
\end{proof}

\begin{lem} \label{conds} Let $F \in \cO_{K}[x,y]$ be an irreducible binary cubic with corresponding family of Frey curves $E_{x,y}$. Write $j(x,y)$ for the $j$-invariant of $E_{x,y}$. Let $E/K$ be an elliptic curve whose conductor $\mathcal N$ is supported on the set $S_F$.  If $j(E)=j(x_1,y_1)$ for some some $x_1,y_1 \in \cO_{K}$ that are coprime outside $S_F$, then $F(x_1,y_1) \in \cO_{K,S_F}^*$. 	\end{lem}

\begin{proof}
	Suppose that $j(E)=j(x_1,y_1)$ for some $x,y \in \cO_{K}$ that are coprime outside $\mathcal H_K$. Assume that $F(x_1,y_1) \not \in \cO_{K,S}^*$. There is a prime $\mP \not \in S_F$ such that $\mP \mid F(x_1,y_1)$. From the explicit equation \eqref{jxy} for $j(x_1,y_1)$ and the resultant identity from Proposition \ref{resultant} we deduce that $\mP$ (since it does not divide $\Delta_F$ by definition) cannot divide $H(x_1,y_1)$, so $v_{\mP}(j(E)) < 0$. This implies that $E$ has potentially multiplicative reduction at $\mP$ and hence $\mP \mid \mathcal N$, a contradiction.
\end{proof}

\section{An effective Chebotarev theorem}

In this section, we extend the main result in Section 7 of \cite{bendah} to general number fields. More precisely, given two elliptic curves $E_1,E_2$ defined over $K$ such that the $G_K$ modules $E_1[2]$ and $E_2[2]$ are not isomorphic, we would like to effectively bound the norm of a prime ideal $\ml$ such that traces of Frobenii $a_{\ml}(E_1)$ and $a_{\ml}(E_2)$ are distinct. We follow the exposition of the  aforementioned section in \textit{loc.\ cit}. 

Given a Galois extension of number fields $L/K$ and a prime ideal $\mathfrak l$ of $K$ which is unramified in $L/K$, we write $\left[ \frac{L/K}{\ml} \right]$ for the conjugacy class in $\Gal(L/K)$ consisting of the Frobenius elements at $\mathfrak l$.
\begin{thm}[Theorem 1.1 in \cite{Lagarias79}] \label{lagarias} There is an absolute, effectively computable, constant $A$ such that for every finite extension $K$ of $\mathbb Q$, every finite
	Galois extension $L$ of $K$ and every conjugacy class $C$ of $\Gal(L/K)$, there exists a prime ideal $\mathfrak l$ of $K$ which is unramified in $L$, for which $\norm_{K/\mathbb Q}(\ml)$ is a rational prime such that
	$$\norm_{K/\Q}(\ml) \leq 2d_L^{A} \text { and } \left[ \frac{L/K}{\ml} \right]= C.$$	
\end{thm}

We will need the following result about the norm of the smallest prime ideal in a given ideal class, which is an easy consequence of \cite{sardari}*{Theorem 1.8}.

\begin{thm} \label{sardari} Given a number field $K$ and a finite set of prime ideals $S$ of $\cO_K$, there exists an effective constant $C_{K,S} >0$, depending only on $K$ and the set $S$, such that every ideal class of $K$ contains a prime ideal $\mP \not \in S$ such that $\norm_{K/\Q}(\mP) < C_{K,S}$.	
\end{thm}

Although Theorem 1.8 in \cite{sardari} is stated assuming the Generalised Riemann hypothesis, the statement that we wrote holds without assuming it. Using GRH, the author of \textit{loc.\ cit.} achieves better bounds for the norm of the sought after prime ideal. He obtains a lower bound on the density of primes with norm smaller than a constant $C$ in every ideal class group, observing that by choosing $C$ large enough the number of such primes must be greater than one. We actually have to make sure that we find a prime outside of a fixed set $S$, so we need to increase the constant $C$ such that the number of prime ideals is strictly greater than $|S|$. Under GRH, one can take $C_{K,S} = A \cdot \max(2|S|h_K, (h_K \log(d_K))^2)$ where $A$ is an implicit constant (explicitly computable) in \cite{sardari}*{inequality (3.17)}, $h_K$ is the class number and $d_K$ is the absolute discriminant of the number field $K$. It was communicated to us via e-mail by Sardari that without using GRH, one can obtain a polynomial bound in terms of $|S|d_K$ for the constant $C_{K,S}$.

\begin{thm} \label{efcheb} Given $L/K$ a fixed Galois extension of number fields and a finite set $S$ of prime ideals in $\cO_K$, there exists a constant $A>0$ such that for every conjugacy class $C$ of $\Gal(L/K)$, there is a prime ideal $\ml \not \in S$ of $\cO_{K}$  for which $$ \norm_{K/\mathbb Q}(\ml) \leq A \text{ and } \left[ \frac{L/K}{\ml} \right] = C.$$
	The constant $A$ is explicitly computable and depends only on $L,K$ and $S$.
	
\end{thm}

\begin{proof}
	Let $\Sprime$ be the subset of $S$ consisting of all the prime ideals from $S$ that do not ramify in $L/K$. If $\Sprime$ is empty, the conclusion follows by applying Theorem \ref{lagarias} to the extension $L$ directly.  Define the ideal $\cm = \prod_{\mP \in \Sprime} \mP $. By Theorem \ref{sardari}, we know that there is a constant $C_{K,S}$ and a prime ideal $\mathfrak a$ of norm less than $C_{K,S}$ that lies in $[\cm]^{-1}$, the inverse class of $\cm$ in the ideal class group of $K$. If we denote by $t \in \cO_K$ a generator of the principal ideal $\cm \cdot \mathfrak a=(t)$, then the quadratic extension $K(\sqrt{t})/K$ is unramified at the primes not dividing $2t$. The norm of its discriminant is bounded in terms of $K$ and the set $S$.
	
	The extensions $L/K$ and $K(\sqrt t)/K$ are Galois. We have the inclusion $ K \subseteq L\cap K(\sqrt t) \subseteq K(\sqrt t)$, which together with the fact that $L\cap K(\sqrt t) / K$ is unramified at $\mathfrak a$ implies that $L \cap K(\sqrt t) = K$. As a consequence, the compositum  $\Lprime := L K(\sqrt t)$ is such that
	$$L'/K \text{ is Galois and } \Gal(L'/K) \cong \Gal(L/K) \times \Gal(K(\sqrt t)/K).$$
	
	Let us now pick $g_{t}$, the non-identity element of the group $\Gal(K(\sqrt{t})/K)$. Applying Theorem \ref{lagarias} above, one obtains a prime ideal $\ml$ of $\cO_K$ such that $\ml$ is unramified in $\Lprime/K$, $\norm_{K/ \Q}(\ml) \leq 2 \left( d_{\Lprime} \right)^A$ and
	$$ \left[\frac{\Lprime/K}{\ml}  \right] = C \times g_{t} \text{ as a conjugacy class of } \Gal(\Lprime/K).$$
	
	Firstly one observes that since $\Lprime/K$ is ramified at the primes in $S$, the ideal $\ml$ does not belong to  $S$.  Also, $\ml$ does not ramify in the extension $L/K$ and $$\left[ \frac{L/K}{\ml} \right]=C.$$ Finally, as a consequence of the formula $d_{\Lprime} = d_{L}^{2} \cdot \norm_{L/\Q}\left( \Delta_{\Lprime/L} \right)$, the absolute discriminant $d_{\Lprime}$ depends on the fields $K$, $L$ and the primes in the set $S$ and can be computed effectively.

\end{proof}
Using the theorem above, it becomes immediately clear that \cite{bendah}*{Proposition 7.4} holds general number fields. For brevity, we include its proof here.
\begin{prop} \label{propcheb} Let $p$ be a rational prime and $E_1/K$ and $E_2/K$ elliptic curves with conductors $\cN_1$ and $\cN_2$, respectively, where $\cN_1 \mid \cN_2$. Write $\rho_i = \overline{\rho}_{E_i,p}$ for $i=1$ and $2$.  Suppose that $\rho_2$ is unramified outside primes dividing $pN_1$ and that $\rho_2$ is irreducible. If $\rho_1 \not \simeq \rho_2$, then there exists a prime $\mathfrak l \subset \cO_K$ with $\mathfrak l \nmid p\cN_1$, for which both
	$$\Tr(\rho_1(\Frob_{\mathfrak l}))\not \equiv \Tr(\rho_2(\Frob_{\mathfrak l})) \pmod{p}$$
	and
	$$l < A$$
	where $l$ is the rational prime $\mathfrak l$ lies over and $A$ is an effectively computable constant depending only on $K, \cN_1$ and $p$. 	
\end{prop}

\begin{proof}
	Consider the (continuous) homomorphism
	$G_K \to \GL_2(\F_p) \times \GL_2(\F_p)$
	given by
	$ \sigma \mapsto (\rho_1(\sigma), \rho_2(\sigma)).$
	Denote by $H$ its image and by $L$ the fixed field of its kernel. Then $L/K$ is finite Galois, unramified outside the set of primes dividing $p \cN_1$ and $\Gal(L/K) \cong H$.
	
	Brauer-Nesbitt together with the classical Chebotarev's density theorem guarantees the existence of such a prime $\mathfrak l$ whose Frobenius at $\mathfrak l$ is an element $(a,b) \in H$ such that
	$$\Tr(a) \not \equiv \Tr(b) \pmod{p}$$
	and by using $S = \{ \mP \subseteq \cO_K \mid \mP \text{ is prime and } \mP \mid p\cN_1 \}$ in Theorem \ref{efcheb} above, one gets the desired bound on $l$.
	
\end{proof}	

\section{The proof of Theorem \ref{GenNf}}
\label{seccomplex}
We would like to emphasize that the idea of considering a prime $\fq$ such that $\fq || \Delta_F$ and the computations carried out in \eqref{FoverFq} - \eqref{valqofj} are due to Bennett and Dahmen \cite{bendah}*{Appendix A.2}, who worked out the particular case $K=\Q$.

Our hypotheses imply that there are $a,b \in \cO_K$ such that
\begin{equation} \label{FoverFq}
F(x,y)= \alpha_0(x-ay)^2(x-by) \pmod{\fq}, \text{ for all } x, y \in \mathcal O_K.
\end{equation}
It is important to observe that $a \not \equiv b \pmod{\fq}.$ Indeed, suppose that is not the case and $a \equiv b \pmod{\fq}$. By using a linear translation, we can assume that $a \equiv b \equiv 0 \pmod{\fq}$, which is equivalent to $\alpha_1 \equiv \alpha_2 \equiv \alpha_3 \equiv 0 \pmod \fq$. Using the formula
\begin{equation} \label{disc3}
\Delta_F = 18 \alpha_0\alpha_1\alpha_2 \alpha_3 + (\alpha_1 \alpha_2)^2 - 27 (\alpha_0 \alpha_3)^2 - 4\alpha_0 \alpha_2^3 - 4 \alpha_1^3 \alpha_3
\end{equation}
for the discriminant of a binary cubic, this would imply that $\fq^2 \mid \Delta_F$, which is excluded in our hypothesis. So $a \not \equiv b \pmod \fq$ indeed. The formula in \eqref{FoverFq} implies that
\begin{equation} \label{HoverFq}
H(x,y) \equiv - \alpha_0^2 (a-b)^2(x-ay)^2 \pmod{\fq}, \text{ for all } x,y \in \mathcal O_K.
\end{equation}

Suppose that $x_0,y_0, z_0 \in \cO_K$ is a solution to \eqref{superelliptic} such that $\gcd(x_0,y_0,z_0)$ is supported only on $S_F$ and $\fq \nmid z_0$. We will prove that the Frey curve $E := E_{x_0,y_0}$ constructed as in \eqref{freyxy} has potentially multiplicative reduction at $\fq$. The discriminant
$\Delta(x_0,y_0)= 2^4 \cdot \Delta_F \cdot F(x_0,y_0)^2$ is clearly divisible by $\fq$. The $j$-invariant of $E$ can be expressed as
\begin{equation} \label{jinvar3}
j(x_0,y_0)= - \frac{2^8 \cdot H(x_0,y_0)^3}{\Delta_F \cdot F(x_0,y_0)^2}
\end{equation}

If $\fq \mid H(x_0,y_0)$, then from \eqref{HoverFq} we get that $\mathfrak q \mid x_0-ay_0$. But this would imply that $\fq \mid F(x_0,y_0)=z_0^l$, which is not allowed. Therefore, $\fq \nmid H(x_0,y_0)$ and
\begin{equation} \label{valqofj} v_{\mathfrak q}(j(x_0,y_0))= -1 - 2l \cdot v_{\mathfrak q}(z_0) =-1.
\end{equation}
This means that, in particular, $E$ has potentially multiplicative reduction at $\fq$.

Write $\overline{\rho}_{E,l}$ for the residual Galois representation $\overline{\rho}_{E,l}: G_K \to \Aut(E[l]) \cong \GL_2(\F_l)$ induced by the action of $G_K$ on the $l$-torsion of $E$. We prove that $\rhobar_{E,l}$ satisfies the hypothesis of Serre's conjecture starting by proving its absolute irreducibility. 

Let $L$ be the Galois closure of $K$. Denote by $\cO_L$ the ring of integers of this number field and by $\fq_L$ a prime above $\fq$. The base change of $E$ to $L$ has potentially multiplicative reduction at $\fq_L$ and \cite[Proposition 6.1]{ASiksek} guarantees the existence of a constant $B_{L,\fq_L}$ such that if $l> B_{L,{\fq_L}}$ the restriction  $\rhobar_{E,l}|_{G_L} : G_L \to \GL_2(\F_{l})$ is irreducible. Eventually increasing $l$ such that $l > v_{\fq_L}(\Delta_F \mathcal O_L)$, from the formulas (\ref{jinvar3}, ~\ref{valqofj}) we see that $l \nmid v_{\fq_L}(j(x_0,y_0))$. Using Lemma 5.1 in \cite{ASiksek}, we obtain that $l \mid \# \rhobar_{E,l}(I_{\fq_L})$, where $I_{\fq_L} \leqslant G_L$ is the inertia subgroup corresponding to $\fq_L$. It is known that every irreducible subgroup of $\GL_2(\F_l)$ which has an element of order $l$ contains $\SL_2(\F_l)$.

As a consequence of the Weil pairing we have that $\det(\rhobar_{E,l}) = \chi_l$, the mod $l$ cyclotomic character. By eventually increasing $l$ such that $L \cap \mathbb Q(\zeta_l) = \Q$, we can ensure that $\det(\rhobar_{E,l}|_{G_L})$ is surjective and together with the observations above this implies the surjectivity of $\rhobar_{E,l}|_{G_L}$. Running through all the prime ideals of $\cO_L$ above $\fq$ we observe that there exists a constant $B_{K,\fq}$ that depends only on $K$ and $\fq$, such that  if $l> B_{K,\fq}$ then $\rhobar_{E,l}$ is surjective.

Our condition that $\gcd(x_0,y_0,z_0)$ is supported on primes contained in $S_F$ implies that if $\mP \notin S_F$ divides $F(x_0,y_0)$, then $\mP \nmid \gcd(x_0,y_0)$. By Proposition \ref{semaway} we see that $E$ is semistable at such primes $\mP$.  From results in \cite{Serre} it follows that the mod $l$ Galois representation $\rhobar_{E,l}$ is unramified away from $S_F \cup \{\mathfrak l | \mathfrak l \subseteq \cO_K \text{ is prime and } \mathfrak l \mid l \}$. In addition, at every prime $\ml \mid l$ the valuation of the discriminant of $E$ 
$$v_{\ml}( \Delta(x_0,y_0))=l \cdot v_{\ml}(z_0) \equiv 0 \pmod{l}.$$
This congruence translates into the technical condition that $\overline{\rho}_{E,l}$ is finite flat at $\ml$, required in the hypothesis of Conjecture \ref{conj:Serre}.
The Serre conductor $\cN$ (prime to $l$ part of its Artin conductor) of this representation is supported only on primes in $S_F$. We also know that $\mathcal N$ divides the conductor of $E$, therefore we can bound the exponent of $\mathfrak a $ in $\mathcal N$ using \cite[Theorem IV.10.4]{silver2}. We get
$$ v_{\mathfrak a}(\mathcal N) \leq 2 + 3v_{\mathfrak a}(3)+ 6 v_{\mathfrak a}(2) \leq 2+6\cdot |K:\Q|$$
for all prime ideals $\mathfrak a \in S_F$. The essential fact is that $\cN$ belongs to a finite set that depends only on the form $F$ and, of course, $K$.

The Galois representation $\rhobar_{E,l}$ satisfies all the hypothesis of Conjecture \ref{conj:Serre} and hence the latter implies the existence of a weight $2$ mod $l$ eigenform $\theta$ over $K$ of level $\cN$, such that for all primes $\mP$ coprime to $l \cN$, we have
$$\Tr(\rhobar_{E,l}(\Frob_{\mP})) = \theta(T_{\mP}).$$
Since there are only finitely many possible levels $\cN$ and the integral cohomology subgroups of $Y_0(\mathcal N)$ are known to be finitely generated, one can conclude that there is a constant $C_1$ that depends only on $K$ and the set $S_F$ such that by taking $l> C_1$ the cohomology subgroups $H^{i}(Y_0(\mathcal N), \Z)$ have trivial $l$ torsion for every $i \geq 1$. This implies that the $l$-torsion of every $H^{i}(Y_0(\mathcal N), \Z_{(l)})$ is trivial for all $i\geq 1$, hence we can guarantee that there exists a weight $2$ complex eigenform $\ff$ with level $\cN$ that is a lift of $\theta$ as explained in Section 2. This is the only ineffective step in our theorem, in the sense that although one can use algorithms to compute $C_1$ for an individual level $\mathcal N$ as it was done in \cite{turcas2018}, we do not know how to write down a formula for the constant $C_1$ in terms of $\mathcal N,F$ and $K$.

Since $\cN$ belongs to a finite set, the list of such possible eigenforms $\ff$ is finite and depends only on $K$. It follows from \cite[Lemma 7.2]{ASiksek} that there is a constant $C_2$ such that if we make sure that $l > C_2$, then the Hecke eigenvalues of $\ff$ belong to $\Z$. 
By Conjecture \ref{conj:ES}, $\ff$ corresponds to an elliptic curve $E_{\ff}$ of conductor $\cN$ or to a fake elliptic curve $A_{\ff}$ of conductor $\cN^2$. We observed earlier in this proof that $l \mid \# \rhobar_{E,l}(I_{\mathfrak q})$, therefore using Proposition \ref{fakeinertia}, we see that for $l > 24$ the latter situation cannot happen and for such primes $l$, $\ff$ corresponds to an elliptic curve $E_{\ff}$ of conductor $\cN$. It is worth mentioning that $E_{\ff}$ does not depend on the solution $(x_0,y_0,z_0,l)$ of the superelliptic equation \eqref{superelliptic}. On the other hand, for all primes $\mP \nmid l \cN$ we have
\begin{equation*}
\Tr(\rhobar_{E,l}(\Frob_{\mP})) = \Tr(\rhobar_{E_{\ff}}(\Frob_{\mP}))
\end{equation*}
which implies
\begin{equation} \label{conEanfF}
\#E( \cO_K/\mP) \equiv \#E_{\ff}(\cO_K/\mP) \pmod{l}.
\end{equation}
We will now prove that, for $l$ large enough, $E[2]$ and $E_{\ff}[2]$ are isomorphic as $G_K$ modules. From \cite{bendah}*{Proposition 6.8} it follows that, since the binary cubic $F$ is irreducible over $K$, the mod $2$ representation $\rhobar_{E,2}$ is also irreducible. Now if the $G_K$ modules $E[2]$ and $E_{\ff}[2]$ are not isomorphic, by Proposition \ref{propcheb} used with $p=2$ we obtain a prime ideal $\mP \subset O_K$, of norm that is bounded in terms of $K$ and $S_F$ such that $\#E(\mathcal O_K/\mP) \not \equiv \#E_{\ff}(\cO_K/\mP) \pmod{2}$. In particular, $ \#E(\mathcal O_K/\mP) - \#E_{\ff}(\mathcal O_K/\mP)$ is non-zero and bounded above in terms of $K$ and $S_F$ as a consequence of the Hasse bounds. Since $l$ divides $ \#E(\mathcal O_K/\mP) - \#E_{\ff}(\mathcal O_K/\mP)$, we infer that there exists a constant $C_3$ such that if $ l > C_3$ then $E[2]$ and $E_{\ff}[2]$ are isomorphic as $G_K$ modules.

In the terminology used by Fisher in \cite{FisherHes}, the elliptic curves $E=E_{x_0,y_0}$ and $E_{\ff}$ are said to be $2$-congruent. Proposition 6.2 in \cite{bendah} implies that $E_{\ff}$ is isomorphic over $K$ to a Frey curve $E_{x_1,y_1}$ for some $x_1,y_1 \in K$. In fact, by Proposition \ref{notchangeN} we can scale the pair $(x_1,y_1)$ such that $x_1,y_1 \in \mathcal O_K$, $\gcd(x_1,y_1) \in \mathcal H_K$ (or is trivial if the class group of $K$ is) and the conductor of $E_{x_1,y_1}$ is still supported only on $S_F$. 
We now make use of Lemma \ref{conds} to get that $F(x_1,y_1) \in \mathcal O_{K,S_F}^{*}$, a contradiction to \eqref{unitseq}.

All of the constants defined in this section depend only on $F$ and $K$, therefore if we choose $A_{K,F}$ to be larger than all of them the proof of our theorem is completed.

\section{$K$ totally real and the proof of Theorem \ref{treal}}
\label{secreal}

When $K$ is totally real, recall that an elliptic curve $E$ defined over $K$ is \textit{modular} if there exists a Hilbert cuspidal eigenform $\ff$ of parallel weight $2$, with rational Hecke eigenvalues, such that there is an isomorphism of compatible Galois representations \begin{equation}\rho_{E,l} \cong \rho_{\ff,l}. \end{equation}
The left-hand side of the above is the Galois representation arising from the action of $G_K$ on the $l$-adic Tate module $T_l(E)$, while the right-hand side is the Galois representation associated to $\ff$ by Taylor in \cite{taylorhilbert}. We are going to make use of the known fact \cite{Freitas2015}*{Theorem 5} that if an elliptic curve $E/K$ is not modular, then its $j$-invariant belongs to a finite set $\cW_K$ that depends only on the base field $K$. Since the finiteness of $\cW_K$ is obtained by applying Falting's theorem to curves of genus greater than one, unfortunately we cannot find the points in $\cW_K$ nor the cardinality of this set.

As it is anticipated in the title, we dedicate this section to proving Theorem \ref{treal}. Before we start the actual proof, we need the following lemma.

\begin{lem} \label{modular} Let $F \in \cO_K[x,y]$ be an irreducible binary cubic. There is a constant $C:= C_{K,F}>0$, depending only on $F$ and on the field $K$, such that the following statement holds:
	
	For all $x,y \in \mathcal O_K$, if there exists a prime $\mP \notin S_F$ such that $v_{\mP}(F(x,y)) \geq C$ and $\mP \nmid \gcd(x,y)$ then the Frey curve $E_{x,y}$ constructed as in \eqref{freyxy} is modular.
\end{lem}

\begin{proof} From the irreducibility of $F$ it follows that $\Delta_F \neq 0$ and $F(x,y) \neq 0$, hence the elliptic curve $E_{x,y}$ is well-defined. Suppose that the curve $E_{x,y}$ is not modular. Without losing generality we can assume that $v_{\mP}(y)=0$. Let $H$ be the Hessian of $F$. Recall the formula \eqref{jxy} for the $j$-invariant from which 
	$$H(x,y)^3= - 2^{-8} \cdot \Delta_F \cdot j(x,y) \cdot F(x,y)^2 $$
	and therefore
	$$v_{\mP}(H(x,y)) = 2 v_{\mP}(F(x,y))/3 + v_{\mP}(j(x,y))/3.$$
	
	Since $j(x,y)$ belongs to the finite set $\cW_K$, we find that there exists a constant $B$, that depends only on $K$ such that $v_{\mP}(j(x,y))/3 \geq B$. Now, if we set $C := \max(1, -B/2+1)$, we observe that $\min(v_{\mP}(F(x,y)), v_{\mP}(H(x,y))) \geq 1.$ 
	
	Using the resultant identity in Proposition \ref{resultant}, we can see that $\mP \mid \Delta_F$ and therefore $\mP \in S_F$, a contradiction.   	
\end{proof}

Having all of this, let us get back to the proof of Theorem \ref{treal}.

Suppose $x_0,y_0,z_0 \in \mathcal O_K$ is a solution to the generalised superelliptic equation \eqref{superelliptic} and that $\gcd(x_0,y_0,z_0)$ is supported only on primes in $S_F$. To the pair $(x_0,y_0)$ we can attach an elliptic curve curve $E:= E_{x_0,y_0}$ as in \eqref{freyxy}.

The hypothesis of the theorem implies that there exists a prime ideal $\mP \notin S_F$ such that $\mP \mid z_0$. As $F(x_0,y_0)= z_0^l$, we know that $\mP \nmid \gcd(x_0,y_0)$ and we can apply Lemma \ref{modular} to see that for $l> C_1$, a constant depending only on $K$ and $F$, the elliptic curve $E$ is modular. Denote by $\cN$, the conductor of our elliptic curve. $\cN$ depends on the solution $x_0,y_0$, as $E$ does.

A major step in this proof is obtainining a Hilbert modular form $\ff$ of parallel weight $2$ whose $l$-adic Galois representation matches the one coming from the $l$-adic Tate module of $E = E_{x_0,y_0}$ such that the level of the form $\ff$ does not depend on the putative solution $x_0,y_0$. Such an object will arrise after applying Theorem 7 of  \cite{AFreitasSiksek}. The latter is a level lowering result, obtained from the combined works of Fujiwara, Jarvis and Rajaei, whose hypothesis requires that the residual Galois representation $\overline{\rho}_{E,l}$ is irreducible.

Proposition \ref{semaway} implies that for $l$ large enough such that it is not supported on primes in $S_F$, the elliptic curve $E$ is semistable at all primes above $l$. By eventually increasing $l$, we can assume that $l$ does not ramify in $K$.
Irreducibility of $\overline{\rho}_{E,l}$ follows from Theorem 2 of \cite{FreSikCriteria}. To be precise, from the just mentioned theorem it follows that there exists an explicit constant $C_2$, depending only on the number field $K$, such that for $l > C_2$, the representation $\overline{\rho}_{E,l}$ is irreducible.
For every prime ideal $\mP \notin S_F$ we know from the proposition mentioned at the beginning of this paragraph that the model of $E$ is minimal, semistable at $\mP$ and that $l \mid v_{\mP}(\Delta(x_0,y_0))$. Hence, by \cite{AFreitasSiksek} it follows that there are
\begin{itemize}
	\item a Hilbert modular form $\ff$ of parallel weight $2$ that is new at level $$\mathcal N_l = \prod\limits_{\mP \in S_F} \mP^{v_{\mP}(\mathcal N)},$$
	\item some prime ideal $\omega$ of the number field $\Q_{\ff}$ generated by the Hecke eigenvalues of $\ff$, such that $\omega \mid l$ and $\overline{\rho}_{E,l} \cong \overline{\rho}_{\ff,\omega}$.
\end{itemize}

As we discussed in the previous section, since $\cN_l$ divides the conductor $\cN$ of the elliptic curve $E$, the exponents of the primes dividing $\cN_l$ are bounded. The possible levels $\cN_l$ belong to a fixed finite set, hence $\ff$ belongs to a finite set of Hilbert modular forms, a set that depends only on the field $K$. From Lemma 7.2 in \cite{ASiksek} it follows that there exists a constant $C_3$, depending only on $K$ such that if $l > C_3$ then $\ff$ must have rational Hecke eigenvalues eigenvalues, i.e. $\Q_{\ff} = \Q$. For such a rational eigenform $\ff$, Conjecture \ref{conj:Serre} implies the existence of an elliptic curve $E_{\ff}$ of conductor $\cN_l$ that corresponds to $\ff$. In particular, for all primes $\mP \nmid l \cN_l$ we have that
$$\Tr(\rhobar_{E,l}(\Frob_{\mP})) = \Tr(\rhobar_{E_{\ff}}(\Frob_{\mP})),$$
which is equivalent to
$$\#E( \cO_K/\mP) \equiv \#E_{\ff}(\cO_K/\mP) \pmod{l}. $$
The reader should be aware that the final part of this proof is identical to the one presented at the end of Section \ref{seccomplex}. As it was pointed out previously, the irreducibility of $F$ implies that $\rhobar_{E,2}$ is irreducible.  Using Proposition \ref{propcheb} we observe that there exists a constant $C_4$ such that if $l> C_4$, then $E[2]$ and $E_{\ff}[2]$ are isomorphic as $G_K$ modules. Using the same result as in the previous section, namely \cite{bendah}*{Proposition 6.2}, we get that $E_{\ff}$ is isomorphic over $K$ to a curve in our Frey-Hellegouarch family, $E_{x_1,y_1}$ for some $x_1,y_1 \in K$. As explained in Proposition \ref{notchangeN}, we can scale the pair $(x_1,y_1)$ such that it becomes integral, $\gcd(x_1,y_1) \in \cH_K$ (it can be made trivial if $K$ has class number one) and the conductor of the Frey curve $E_{x_1,y_1}$ remains supported only on the primes in $S_F$. We now get a contradiction to the hypothesis of our theorem, since Lemma \ref{conds} implies $F(x_1,y_1) \in \cO_{K,S_F}^{\times}$.

All four constants defined in this section depend only on the form $F$ and the number field $K$. We conclude the proof of the theorem by choosing the constant $A_F$ to be greater than all these constants.

\medskip


\noindent \textbf{Remark.} The only ``ineffective" step in this section is the application of Lemma \ref{modular}, which guarantees that for $l$ large enough, the Frey-Hellegouarch curves we care about are modular.

\subsection*{Acknowledgements}
We are indebted to the anonymous referee for carefully reading this paper and for many helpful remarks. The author is very grateful to his advisor Samir Siksek for suggesting the problem and for his great support. It is also a pleasure to thank  Oleksiy Klurman, Vlad Matei, Naser Sardari and Haluk \c Seng\" un for useful discussions.


\normalsize

\bibliographystyle{amsplain}
\bibliography{perf-pow}

\end{document}